\documentclass[11pt]{amsart} 

\usepackage[utf8]{inputenc}
\usepackage[T1]{fontenc}
\usepackage[english]{babel}
\usepackage{indentfirst}
\usepackage{amssymb, amsmath,amsfonts,amsthm,mathrsfs}
\usepackage{latexsym}
\usepackage{wasysym}
\usepackage{enumerate}
\usepackage{xypic}
\usepackage{graphicx}
\usepackage{pgf, tikz}
\usepackage{booktabs}
\usepackage[normalem]{ulem}
\usepackage[makeroom]{cancel}
\usepackage{empheq}

\usepackage{hyperref}
\hypersetup{
	colorlinks,
	citecolor=black,
	filecolor=black,
	linkcolor=black,
	urlcolor=black
	linktoc=all,     }
 
 \usepackage{fancyhdr}
 \theoremstyle{plain}
 \newtheorem{thm}{Theorem}[section]
   \newtheorem*{thm*}{Theorem}
 \newtheorem*{thm4.5}{Theorem 4.5}

 \newtheorem{lem}[thm]{Lemma}
 
 \newtheorem{prop}[thm]{Proposition}
 \newtheorem*{prop*}{Proposition}

 \newtheorem{conj}[thm]{Conjecture}
 \newtheorem*{conj*}{Conjecture}

 \theoremstyle{definition}
 \newtheorem{defi}[thm]{Definition}

 \newtheorem{oss}{Remark}
 \newtheorem*{Notation}{Notation}

 \theoremstyle{remark}

\newcommand{\Z}{\mathbb{Z}}

\newcommand{\F}{\mathbb{F}}

\newcommand{\stab}{\mathrm{stab}}

\newcommand{\GL}{\mathrm{GL}}

\def\cP{\mathcal{P}}
\def\cF{\mathcal{F}}

\begin{document}

\begin{titlepage}
\title {M\"obius  function of the subgroup lattice of a finite group and Euler characteristic  }
\author {Francesca Dalla Volta}
\address{Francesca Dalla Volta: Dipartimento di Matematica e Applicazioni, Universit\`a degli Studi di Milano-Bicocca, via Roberto Cozzi 55, 20125 Milano, Italy}
\email {francesca.dallavolta@unimib.it}
\author{Luca Di Gravina} \thanks{The authors are members of the Italian National Group for Algebraic and Geometric Structures and their Applications (GNSAGA - INdAM) and  thank INdAM  for the support.  The second author is also a member of the research training group {\it GRK 2240: Algebro-Geometric Methods in Algebra, Arithmetic and Topology}, funded by DFG}\address {Luca Di Gravina: Mathematisches Institut, Heinrich-Heine-Universit\"at  D\"usseldorf, 40225 D\"usseldorf, Germany
}
\email{luca.di.gravina@hhu.de}
\end{titlepage}
	
\begin{abstract}
The M\"{o}bius function of the subgroup lattice of a finite group
has been introduced by Hall  and applied to investigate several questions.
In this paper, we consider the M\"obius function defined on an order ideal related to the lattice of the subgroups of an irreducible subgroup $G$ of the general linear group $\GL(n,q)$ acting on the $n$-dimensional vector space $V=\mathbb{F}_q^n$, where $\F_q$ is the finite field with $q$ elements. 
We find a relation between this function  and the Euler characteristic of two simplicial complexes $\Delta_1$ and $\Delta_2$, the former  raising from the lattice of the subspaces of $V$, the latter from the subgroup lattice of $G$.
\end{abstract}

\maketitle	
	
\begin{small}

\noindent{\bf Keywords:} M\"obius function, subgroup lattice, linear groups, Euler characteristic, simplicial complexes. \smallskip

\noindent{\bf 2020 MSC:} 20B25, 20D60, 05E16, 05E45 \smallskip

\noindent {\bf Corresponding author:} Francesca Dalla Volta, \email{francesca.dallavolta@unimib.it}

\end{small}

\setcounter{page}{1}
\section{Introduction}

In this paper, motivated by recent and less recent results about  the M\"obius function $\mu $  for the subgroup lattice $\mathcal{L}(G)$ of a finite group $G$, we give a
result which  relates the M\"obius function for a subgroup $G$ of $\GL(n,q)$ to two  simplicial complexes: one  defined from the lattice of the subspaces fixed by a reducible subgroup $H\leq G$ and the second from the lattice $\mathcal {L}(G)$ of the subgroups of $G$.
 
We will introduce in the Preliminaries all the definitions and details useful for reading the paper.

In his PhD thesis \cite{shareshian96}, Shareshian  
considers the problem of computing  $\mu(1,G)$ for several finite classical groups $G$; the idea is to  approximate $\mu(1,G)$  through a \emph{good function} $f_{G,n,p}(u,1)\,$, such that: 
\begin{equation}\label{eq_MobShareshian1}
\mu(1,G(n,\,p^u))=f_{G,n,p}(u,1)+\sum_{K\in\mathcal{C}_9}\mu(1,K)\,. 
\end{equation} 
Here, $G=G(n,p^u)$ denotes a family of finite classical groups with the same defining classical form, which act in a natural way on the vector space $V$ of finite dimension $n$ over the finite field $\F_q$  of order $q=p^u\,$. If  $\,\mathcal{C}_1,\dots,\mathcal{C}_8,\mathcal{C}_9\,$ are the Aschbacher classes of maximal subgroups of a finite classical group (see \cite{kleidlieb90}), $\mathcal{C}_9$  is the class of almost simple groups not belonging to the first $8$ classes of  ``geometric type'' and the function $f_{G,n,p}(u,1)\,$ provides an estimate of $\mu(1,G)$ with respect to the contributions given by the subgroups of $G$ which belong to the classes $\mathcal{C}_i\,$, for $i\in\{1,\dots,8\}$. 

Actually, Shareshian's approach focuses on the first class $\mathcal{C}_1(G)$, that is the class of reducible subgroups of $G$. 

In particular, the reducible subgroups of $G$ contribute to $f_{G,n,p}(u,1)$ through the computation of the M\"obius function of
\begin{align*}
\widehat{\mathcal{I}}_1(G) &:=\{ K\leq G\mid  K\leq M \text{ for some } M\in \mathcal{C}_1(G) \} \cup \{G\}\,,
\end{align*}
which is obtained by adjoining the maximum $G$ to the order ideal
$${\mathcal{I}}_1(G):=\{ K\leq G\mid  K\leq M \text{ for some } M\in \mathcal{C}_1(G) \}\,,$$ 
 that is
$$\mu_{\widehat{\mathcal{I}}_1(G)}(1,G)=  -\sum_{\substack{K\in\mathcal{I}_1(G)}}\mu(1,K)\,$$ \\
 and \begin{align}\label{eq_MobShareshian2}
\mu(1,G)=\mu_{\widehat{\mathcal{I}}_1(G)}(1,G)  -\sum_{\substack{K<G \\ K\notin\mathcal{I}_1(G)}}\mu(1,K)\,.
\end{align}

In this paper, we will consider 
irreducible subgroups $G$ of the general linear group $\GL(n,q)$, that is groups of linear automorphisms of a vector space $V$ of dimension $n$ over the finite field $\F_q$ with $q$ elements, which  fix no non-trivial subspace of $V$. In this hypothesis, we will take a reducible subgroup $H$ of $G$ (that is, $H$ fixes some proper subspace of $V$) and we will work on  the analogue of $\mu_{\widehat{\mathcal{I}}_1(G)}(1,G)$, namely  $\mu_{\widehat{\mathcal{I}}(G,H)}(H,G)$, so that \begin{align}\label{eq_MobShareshianH}
\mu(H,G)=\mu_{\widehat{\mathcal{I}}(G,H)}(H,G)  -\sum_{\substack{K\notin\mathcal{I}(G,H)\\ H\leq K<G}}\mu(H,K)\, 
\end{align} (see Sections 2 and 3 for all precise definitions).

The subject of this paper is somehow motivated by the following conjecture:

\begin{conj}[Mann, \cite{mann05}]\label{MannConj} 
	Let $G$ be a PFG group and $\mu$ the M\"obius function on the lattice of open subgroups of $G$. Then $|\mu(H,G)|$ is bounded by a polynomial function in the index $|G:H|$ and  the number of subgroups of $G$ of index $m$ with $\mu (H,G)\neq 0$ grows at most polynomially in $m$. 
\end{conj} 

Indeed, although the problem is still open in its general setting, it was reduced by Lucchini in \cite{lucchini10} to the study of similar growth conditions for finite almost-simple groups.\\

The following Theorem is the main result of the present paper:
\begin{thm4.5}\label{thm4.5}
Consider a vector space $V$  of finite dimension over $\F_q\,$. 
Let $G$ be an irreducible subgroup of $\,\mathrm{GL}(V)$ and $H\leq G$. 
Then 
\begin{equation}
-\mu_{\widehat{\mathcal{I}}(G,H)}(H,G)=\sum_{E\in\Psi'(G,H)}(-1)^{|E|}\,=\sum_{X\in\Psi(G,H)}(-1)^{|X|}\,=-\tilde{\chi}(\Delta _1 )=-\tilde{\chi}(\Delta _2).
\end{equation}
\end{thm4.5}
${\widehat{\mathcal{I}}(G,H)}$,  ${\Psi'(G,H)}$, ${\Psi(G,H)}$, $\Delta _i$, and $\tilde{\chi}(\Delta _i)$ (for $i=1,\,2$) are defined in the following Section 3, also for irreducible subgroups $H$. This will allow us to avoid the restriction to only reducible subgroups $H$ of $G$ in the statement of the Theorem.

In the final Section, we will use Theorem \ref{RevShar3.27} to deal with $\mu (H,G)$ in some particular case.
In \cite{DG2022},  Theorem \ref{RevShar3.27} is used to attack Conjecture \ref{MannConj} for some class of subgroups $H$ of linear and projective groups.\\

We thank Andrea Lucchini and Johannes Siemons for many useful discussions.

\section{Preliminaries}
In this paper all the groups and sets are finite.

For  main results about posets and lattices, we refer to \cite{stanley86}. Here, we just recall some basic fact, useful  for  reading  the paper.
 \begin{defi}
Let $\mathcal{P}$ be a finite poset.
The {\bf M\"obius function} associated with  $\mathcal{P}$ is the map $\,\mu_{\mathcal{P}}:\mathcal{P}\times \mathcal{P}\rightarrow\Z$ satisfying 
$$\mu_{\mathcal{P}}(x,y)=0\quad  \text{ unless }\;x\leq y\,,$$
and defined recursively for $x\leq y\,$ by 
\begin{equation}\label{EqDef_MobiusFun1}
\mu_{\mathcal{P}}(x,x)=1\;  \quad\text{ and }\quad\; \sum\limits_{x\leq t\leq y}{\mu_{\mathcal{P}}(x,t)}=0 \;\text{ if }\, x<y\,.
\end{equation}

\end{defi}

\begin{Notation} If $\mathcal{P}$ is the subgroup lattice $\mathcal{L}(G)$ of $G$, we will write $\mu (H,K)$ instead of $\mu _{\mathcal{L}(G)}(H,K)$.
\end{Notation}

\begin{defi} An (abstract) {\bf simplicial complex} $\Delta$ on a vertex set
$T$ is a collection $\Delta$  of subsets of $T$ satisfying the two following conditions:
\begin{itemize}
\item if $t\in T$,  then $\{t\} \in \Delta$;
\item if $F\in \Delta$  and $G\subseteq F$, then $ G\in \Delta $.
\end{itemize}
\end{defi}
An element $F\in \Delta $ is called a face of  $\Delta$, and the dimension
of $F$ is defined to be $|F|-1$. In particular, the empty set $\emptyset $ is a face of $\Delta$ 
(provided $\Delta \neq \emptyset$)
 of dimension $-1$. 

\begin{defi}
The {\bf Euler characteristic} $\chi(\Delta)$ of the simplicial complex $\Delta$ is so defined:
$$\chi(\Delta) :=\, \sum _{i\geq0}(-1)^i\cF _i\,=\,\cF_0 - \cF_1 + \cF_2 - \cF_3 + \dots$$
where  $\cF_i$ denotes  the number of the $i$-faces, i.e., of the faces of dimension $i$,  in $\Delta$. 
\end{defi}
We recall here the definition of an order ideal:
 
\begin{defi}Let $(\mathcal{P},\leq)$ be a poset. An \textbf{order ideal} of $\mathcal{P}$ 
is a subset $I\subseteq \mathcal{P}$ such that 
\begin{align}\label{OrderIdealCondition}
\forall\,x\in I,\; \forall\,t\in \mathcal{P}\qquad t\leq x \,\Rightarrow\, t\in I\,.
\end{align}
In particular, if $A$ is a subset of $\mathcal{P}$, then the set
$$\mathcal{P}_{\leq A} := \{s\in \mathcal{P} \mid s\leq a \text{ for some } a\in A\}\subseteq \mathcal{P}$$
is the order ideal of $\mathcal{P}$ generated by $A$.
\end{defi}

Observe that  a simplicial complex on a set of vertices $T$  is just an order ideal of the boolean algebra $B_T$  that contains all one element subsets of $T$. 
\begin{defi} Let $\Delta$ be a simplicial complex, and  $\chi(\Delta)$ be the Euler characteristic of $\Delta$.  The \textbf{reduced Euler characteristic} $\tilde{\chi}(\Delta)$ of $\Delta$  is defined by $\tilde{\chi}(\emptyset)=0$, and $\tilde{\chi}(\Delta)=\chi(\Delta)-1$ if $\Delta\ne\emptyset$. See also Equation (3.22) in \cite{stanley86}. 
\end{defi}

It is possible to build a simplicial complex from a poset
 $\cP$ in the following way:\\
 
The \textbf{order complex} $\Delta(\cP)$ of $\cP$ is defined as the simplicial complex whose vertices are the elements of $\cP$ and whose $k$-dimensional faces are the chains $a_0\prec a_1\prec\cdots\prec a_k$ of length $k$ of distinct elements $a_0,\ldots,a_k\in\cP$.

Now, denote by $\widehat{\cP}$ the finite poset obtained from $\cP$ by adjoining a least element $\hat{0}$ and a greatest element $\hat{1}$. The M\"obius function $\mu_{\widehat{\cP}}(\hat{0},\hat{1})$ is related to $\tilde{\chi}(\Delta(\cP))$ by a well-known result  by Hall in \cite{hall36} about the computation of $\mu_{\widehat{\cP}}(\hat{0},\hat{1})$ by means of the chains of even and odd length between $\hat{0}$ and $\hat{1}$.

\begin{prop}[\rm see \cite{stanley86}, Proposition 3.8.6]\label{prop:reduced}
Let $\cP$ be a finite poset. Then
$$ \mu_{\widehat{\cP}}(\hat{0},\hat{1})=\tilde{\chi}(\Delta(\cP)). $$
\end{prop}

\section
{The ideal $\mathcal{I}(G,H)$ and the complexes  $\Delta _i$ in Theorem 4.5} \label{section_DefIdealReducible}
In this section, and in the rest of the paper, $G$ is an irreducible subgroup of 
the general linear group $\GL(n,q)$ over the vector space $V=\mathbb{F}_q^n$ of finite dimension $n$ over the finite field $\F_q$ with $q$ elements. We consider the natural action of $G$ on the set of subspaces of $V$.\\
  
We define the order ideal $\mathcal{I}(G,H)$ and the simplicial complexes $\Delta _i$ (for $i=1,2$) that we consider in Theorem \ref{RevShar3.27}. In Remark 2 below, we will explicitly observe that the two complexes $\Delta _i$  are  not the order simplicial complex rising from $\mathcal{I}(G,H)$.\\

Given a subgroup  $G$  of $\GL (n,q)$ and a subgroup 
$H$
 of $G$, put  
$$\mathcal{L}(G)_{\geq H}:=\{K\leq G\,\mid\,H\leq K\}$$
and \begin{align*}
\mathcal{C}(G,H)&:=\{\,\mathrm{stab}_G(W)\mid 0< W< V\,,\,H\subseteq\mathrm{stab}_G(W) \,\}.
\end{align*}

\begin{defi}\label{Def_RedSubIdeal} 
	The \textbf{reducible subgroup ideal} in $\mathcal{L}(G)_{\geq H}$ is the order ideal of $\mathcal{L}(G)_{\geq H}$ generated by $\mathcal{C}(G,H)\,$. Namely,
	\begin{align*}
	\mathcal{I}(G,H)&=\{ K\leq G\mid H\leq K\leq M \text{ for some } M\in \mathcal{C}(G,H) \}\,.
	\end{align*}
\end{defi}

\begin{oss}
	If $H$ is reducible, that is, $H$ fixes some  non-trivial subspace $W$ of $V$, then $H\in\mathcal{I}(G,H)$. Otherwise, if $H$ is irreducible,  $H\notin\mathcal{I}(G,H)$ and $\mathcal{I}(G,H)$ is empty . The subspaces fixed by $H$ are said to be $H$-invariant.
\end{oss}

\begin{defi}
	If $H$ is reducible, we set
	\begin{align*}
	\widehat{\mathcal{I}}(G,H)&:= \mathcal{I}(G,H)\cup\{G\}\,
	\end{align*} 
	by adjoining the maximum $G$ to $\mathcal{I}(G,H)$, which has minimum $H$. Otherwise, if $H$ is irreducible, we set $\widehat{\mathcal{I}}(G,H):=\{H,G\}$ by adjoining the minimum $H$ and the maximum $G$ to the empty poset $\emptyset$.
\end{defi}
\begin{oss} Here we just note that the poset $\mathcal{I}(G,H)$ has already a minimum if $H$ is reducible, so that  $\mu_{\widehat{\mathcal{I}}(G,H)}(H,G)$ is not, in general, the reduced Euler characteristic of the order complex $\Delta({\mathcal{I}}(G,H))$ of $\mathcal{I}(G,H)$.
\end{oss}

To define the simplicial complexes $\Delta _i$ of Theorem 4.5, 
we begin with fixing some more notation. 
We denote by 
$\mathcal{S}(V,H)$ the lattice of $H$-invariant subspaces of $V$
and define
$$\mathcal{S}(V,H)^*:=\mathcal{S}(V,H)\setminus\{0,V\}\,.$$
Moreover, given an irreducible group $G\leq \GL (V)$, and $H\leq G$, we will consider the following three sets:
	\begin{itemize} 
		\item[(a)] $\Psi(G,H):=\{X\subseteq\mathcal{C}(G,H)\,\mid\,\bigcap_{M\in X}M\neq H\,\}\,$;
		\item[(b)] $\Psi(G,H)^\complement:=\{Y\subseteq\mathcal{C}(G,H)\,\mid\, \bigcap_{M\in Y}M= H\,\}\,$;
		\item[(c)]		
		$\Psi'(G,H):=\{E\subseteq\mathcal{S}(V,H)^*\,\mid\,\bigcap_{W\in E}\mathrm{stab}_G(W)\neq H\}\,.$ 
	\end{itemize}
Observe that $\emptyset\in\Psi(G,H)$ and $\emptyset\in\Psi'(G,H)$, but $\emptyset\notin\Psi(G,H)^\complement\,$. 

\begin{oss}
If $H$ is an irreducible subgroup of $G$, then $\mathcal{S}(V,H)^*=\emptyset $ and $\Psi'(G,H)=\{\emptyset\}=\Psi (G,H)$. 
Since for irreducible $H$ we have that 
$\widehat{\mathcal{I}}(G,H)=\{H,G\}$, then Theorem \ref{RevShar3.27} is trivially verified in this case.
\end{oss}

\begin{defi}
The simplicial complexes  $\Delta _i$ of Theorem 4.5 are so defined: 
\begin{enumerate}[]
	\item $\Delta _1$:
	\begin{itemize}
		\item The set of vertices $T_1$ is given by the subspaces $W\in\mathcal{S}(V,H)^*$ for which $H\neq\stab_G(W)$;
		\item the set of faces of $\Delta _1 $ is given by $\Psi'(G,H)$.
	\end{itemize} \medskip
	\item $\Delta _2$:
	\begin{itemize}
		\item The set of vertices $T_2$ is given by the subgroups $M\in\mathcal{C}(G,H)$ such that $H\neq M$;
		\item the set of faces of $\Delta _2 $ is given by $\Psi(G,H)$.
	\end{itemize}	
\end{enumerate}
\end{defi}

We explicitly observe what happens in the special case when, for some proper non-trivial subspace $W$ of $V$, the subgroup $H=\stab_G(W)$ is maximal with respect to the property of being a stabilizer of a proper non-trivial subspace of $V$ in $G$. 
In this case, we note that, by definition, $T_1=T_2=\emptyset$ and the set of faces of $\Delta_1$ and $\Delta_2$ is $\{\emptyset\}$. Then Theorem 4.5 is trivially verified.

\section{Computing $\mu_{\widehat{\mathcal{I}}(G,H)}(H,G)$}
In order to prove Theorem \ref{RevShar3.27}, we need the   following Proposition \ref{RevSharApp2.17}  that gives a link between  $\Psi(G,H)$ and $\Psi'(G,H)$, and also shows that the reduced Euler characteristics of the complexes $\Delta _i$ coincide. The proof of Proposition \ref{RevSharApp2.17} follows at once from Lemma \ref{RevShar2.17}.

\begin{lem}\label{RevShar2.17}
Let $T$ be a subgroup of a finite group $L$  acting on a finite set $X$   
and let $X'\subseteq X$ be a subset such that $T\leq L_x\,$ for all $x\in X'$  (as usual, $L_x$ denotes the stabilizer of $x$ in  $L$). Set 
	\begin{itemize} 
		\item $\mathcal{L}:=\{L_x\,\mid\,x\in X'\}\,$; 
		\item $\mathcal{R}:=\{E\subseteq{\mathcal{L}}\,\mid\, \bigcap_{K\in E}K\neq T\}\,$; 
		\item $\mathcal{S}:=\{Q\subseteq X'\,\mid\, \bigcap_{x\in Q}L_x\neq T\}\,$. 
	\end{itemize} 
Then $$\sum_{E\in\mathcal{R}}(-1)^{|E|}=\sum_{Q\in\mathcal{S}}(-1)^{|Q|}\;.$$ 
\end{lem}

\begin{proof}
	If $Q\in \mathcal{S}$ and  $E\in\mathcal{R}$, set  
	$$\,\mathcal{L}_Q=\{L_x\mid x\in Q\}\;\text{ and }\;
	\mathcal{S}_E=\{Q\in\mathcal{S}\, \mid\, E=\,\mathcal{L}_Q\}.$$ It is immediate to realize that
	$$\mathcal{S}=\bigsqcup_{E\in\mathcal{R}}\mathcal{S}_E$$
	is the disjoint union of all the $\mathcal{S}_E$. 
	So, it suffices to show that for each $E\in\mathcal{R}$ the following identity is verified: 
	\begin{equation}\label{Eq_Lemma1}
	(-1)^{|E|}=\sum_{Q\in\mathcal{S}_E}(-1)^{|Q|}\,.
	\end{equation} 
	For $E=\emptyset$, the identity \eqref{Eq_Lemma1} is trivially true.
	Now, fix a non-empty $E\in\mathcal{R}$, and for each $K\in E$ define
	$$X'_K=\{x\in X'\,\mid\, L_x=K\}\subseteq X'\,.$$
	Let $Q\in\mathcal{S}_E$ and observe that $Q$ can be represented as the following disjoint union:
	$$Q=\bigsqcup_{K\in E}Q_K\,,$$
	where $\,Q_K=\{x\in Q\,\mid\, L_x=K\}\subseteq X'_K\,$ and $Q_K\neq\emptyset$ (this property characterizes the elements $Q$ of $\mathcal{S}_E$).  
	Since  $\sum_{\emptyset\neq Q_K\subseteq X'_K}(-1)^{|Q_K|}{=-1}$ (see Remark \ref{BinomCoeffSets} below),	
	we get 
	\begin{align*}\sum_{Q\in\mathcal{S}_E}(-1)^{|Q|}=\sum_{Q\in\mathcal{S}_E}(-1)^{\sum_{K\in E} |Q_K|}=\prod_{K\in E}\left(\sum_{\emptyset\neq Q_K\subseteq X'_K}(-1)^{|Q_K|}\right)=\prod_{K\in E}(-1)=(-1)^{|E|}\,\end{align*}
	and we obtain the identity \eqref{Eq_Lemma1}.
\end{proof}

\begin{prop}\label{RevSharApp2.17} 
	Let $V$ be a vector space of finite dimension over $\F_q\,$ and consider a subgroup 	 $H\leq G\leq\mathrm{GL}(V)$. 
	We have:
	\begin{equation}\label{eq_RevSharApp2.17}
	\sum_{E\in\Psi'(G,H)}(-1)^{|E|}=\sum_{{X}\in\Psi(G,H)}(-1)^{|{X}|}\,.
	\end{equation}
	Equivalently, $\tilde{\chi}(\Delta _1 )=\tilde{\chi}(\Delta _2)$.
\end{prop}

\begin{proof}
Consider the natural action of $G$ on the set of subspaces of $V$. By Lemma \ref{RevShar2.17} below, the equality follows at once from  the definitions of $\mathcal{S}(V,H)^*$, $\mathcal{C}(G,H)$, $\Psi(G,H)$ and $\Psi'(G,H)\,$. \\
With previous notation, we take $T=H$, $L=G$, $\mathcal{R}=\Psi (G,H)$, $\mathcal{S}=\Psi'(G,H)$.\\
\end{proof}

To prove Theorem \ref{RevShar3.27}, we need Proposition \ref{RevSharApp2.6} which is achieved through Theorem \ref{Crosscut} (Crosscut Theorem, see \cite[Corollary 3.9.4]{stanley86}) and Remark \ref{BinomCoeffSets}.

\begin{oss}\label{BinomCoeffSets}
For every finite set $A$ of cardinality $n>0$ we have  
	$$\sum_{S\subseteq A}(-1)^{|S|}=\sum_{k=0}^n\binom{n}{k}(-1)^k=(1-1)^n=0\,.$$
\end{oss}
\begin{thm}[Crosscut Theorem]\label{Crosscut}  Let $L$ be a finite lattice with minimum
$\hat 0$
and maximum $\hat 1$, so that $\hat 0\neq \hat 1$. Let $M$ be the set of all coatoms in
$L$. Let $X\subseteq L$ be a subset such that $M \subseteq X$ and $\hat1 \notin X$.\\ Given $\mathcal{Y} := \{Y\subseteq  X\,\mid\,  Y\neq {\emptyset}\, \text{ and }\,
\bigwedge_{x\in Y}\,x=\hat 0\}$, the following equality holds:
$$\mu _L(\hat0, \hat 1) =\sum_{Y\in\mathcal{Y}}(-1)^{\mid Y\mid}.$$
\end{thm}

\begin{prop}\label{RevSharApp2.6}
	Let $V$ be a vector space of finite dimension over $\F_q\,$. 
	Let $H\leq G\leq\mathrm{GL}(V)$. 
	Then we have that
	\begin{equation}\label{eq_RevSharApp2.6}
	\mu_{\widehat{\mathcal{I}}(G,H)}(H,G)=\sum_{Y\in\Psi(G,H)^\complement}(-1)^{|Y|}\,.
	\end{equation}	
\end{prop}
\begin{proof} 
We observe that $\widehat{\mathcal{I}}(G,H)\subseteq \mathcal{L}(G)_{\geq H}$,  is a lattice because the join of  two subgroups $K_1,K_2\in\mathcal{I}(G,H)$ is either in $\mathcal{I}(G,H)$ or equal to $G$.  The meet of $K_1,K_2\in\mathcal{I}(G,H)$ is $K_1\cap K_2\,$. Hence, $\widehat{\mathcal{I}}(G,H)$ is a finite lattice, whose set of coatoms is contained in $\mathcal{C}(G,H)$, and its maximum $G\notin\mathcal{C}(G,H)$, because $G$ is assumed to be irreducible.
Since 
$$\Psi(G,H)^\complement=\{Y\subseteq\mathcal{C}(G,H)\,\mid\,Y\neq\emptyset \;\text{ and }\; \bigcap_{M\in Y}M= H\,\}\,,$$ by Theorem \ref{Crosscut} we immediately obtain \eqref{eq_RevSharApp2.6}. 
\end{proof}

Now, observe that the disjoint union $\Psi(G,H)\cup\Psi(G,H)^\complement$ is the power set of $\mathcal{C}(G,H)$, so that by Remark \ref{BinomCoeffSets} we have
\begin{equation}\label{eq_SharAppBinCoeff}
\sum_{X\in\Psi(G,H)}(-1)^{|X|} + \sum_{Y\in\Psi(G,H)^\complement}(-1)^{|Y|} = 0.
\end{equation}

If we put together equations \eqref{eq_RevSharApp2.17}, \eqref{eq_RevSharApp2.6}, and \eqref{eq_SharAppBinCoeff}, we get:

\begin{thm}\label{RevShar3.27}
	Let $V$ be a vector space of finite dimension over $\F_q\,$. 
	Let $H\leq G\leq\mathrm{GL}(V)$. 
	Then 
	\begin{equation}\label{eq_RevShar3.27}
	-\mu_{\widehat{\mathcal{I}}(G,H)}(H,G)=\sum_{E\in\Psi'(G,H)}(-1)^{|E|}\,.
	\end{equation}	
\end{thm}
\begin{proof} 
	By \eqref{eq_SharAppBinCoeff}, we have 		 $$\sum_{{X}\in\Psi(G,H)}(-1)^{|{X}|}=-\sum_{{Y}\in\Psi(G,H)^\complement}(-1)^{|{Y}|}\,.$$  
	Then, by Proposition \ref{RevSharApp2.17} and Proposition \ref{RevSharApp2.6}, 
	\begin{align*} 
	\sum_{E\in\Psi'(G,H)}(-1)^{|E|} =\sum_{{X}\in\Psi(G,H)}(-1)^{|X|} & =-\sum_{{Y}\in\Psi(G,H)^\complement}(-1)^{|{Y}|} \\ \\
	& =-\mu_{\widehat{\mathcal{I}}(G,H)}(H,G)\,.
	\end{align*} 
\end{proof}

\section{Final Remark}
Going back to Conjecture \ref{MannConj}, we observe that the knowledge of
$$\sum_{E\in\Psi'(G,H)}(-1)^{|E|}\,$$
coming from \eqref{eq_RevShar3.27} can be exploited to estimate the value $\mu (H,G)$ of the M\"obius function $\mu$ of $G$ for $H\leq G$, at least in some particular case.  Here we just give an example for a particular reducible $H$, taking  $G=\GL(n,q)$.

Following the idea suggested by Shareshian in \cite{shareshian96},  one could write 
\begin{equation}\label{eq_MobShareshian1b}
\mu(H,G)=f_{G,n,q}(H)+\sum_{K\in\mathcal{C}_9}\mu(H,K)\,,
\end{equation}
where $f_{G,n,q}(H)$ depends on the classes $\mathcal{C}_i(G,H)$, for $i=1,\dots,8$, in Aschbacher's classification.
In some lucky case, $H$ is not contained in maximal subgroups belonging to classes $\mathcal(C_i),\, i\neq 1,\, 9$.
This happens, for example, in the following case: \\
Let $V\simeq\F_q^n\,$ be a vector space of finite dimension $n$ over $\F_q\,$
and fix the following basis of $V\,$:
$$\mathcal{E}:=\{w_1,\dots,w_{m},\, v_{m+1},\dots,  v_{n}\},$$ so that 
\begin{align*} V = \langle w_1,\dots,w_{m}\rangle \oplus\langle v_{m+1},\dots, v_n\rangle\,
\end{align*} 
If  $W=\langle w_1,\dots,w_{m}\rangle$,  $H$ is the subgroup of $\mathrm{GL}(n,q)$ acting as $\mathrm{GL}(m,q)$ on $W$ and fixing all the elements $v_{m+1},\dots, v_n$. 

We do not give in this context the details of the proof of the following Theorem, which needs many technical arguments. In \cite{DGthesis}, all the details are given.

\begin{thm}[\cite{DGthesis}]\label{esempietto}
	Let $G=\mathrm{GL}(n,q)$, and let $H\leq G$ be such that 
	$$H=\mathrm{GL}(m,q)\oplus I_{n-m}\,.$$
	Let $q=p\,$ be an odd prime and let the dimension $n$ be prime. 
	If $\,n-m+1\,$ is prime, then $$\mu_{\widehat{\mathcal{I}}(G)}(H,G)=0=f_{G,n,p}(H)\,$$ 
	so that 
	\begin{equation*}
	\mu(H,G)=\sum_{\substack{K\in\mathcal{C}_9(G,H)\,,\, H\subseteq K}}\mu(H,K)\,.
	\end{equation*}  
\end{thm}

In general, we do not have much information about the ninth class.
Just to give an example, we considered the groups of low dimension studied by Schr\"oder in her PhD thesis (\cite{schroder15}), and we saw that in dimension $n=13$ also class $\mathcal{C}_9(G,H)$ is empty for $p>5$. In this case, $\mu(H,G)=0$. \\

All the details and  data needed to prove Theorem 5.1 are contained in \cite{DGthesis} and are available if requested.


\begin{thebibliography}{49}

	\addcontentsline{toc}{chapter}{Bibliography}
	{\small 
	
		\bibitem{DG2022} L. Di Gravina,   A closure operator on the subgroup lattice of $\GL(n,q)$ and $\mathrm{PGL}(n,q)$ in relation to the zeros of the M\"obius function,  \emph{Journal of Group Theory}, 2023. https://doi.org/10.1515/jgth-2023-0021
		 
		\bibitem{DGthesis} L. Di Gravina,  \emph{Some questions about the M\"obius function of finite linear groups},
		 {Ph.D. Thesis, Universit\`{a} degli Studi di Milano-Bicocca, 2022}.
		 					
		\bibitem{hall36} P. Hall, The Eulerian functions of a group, \emph{Quarterly Journal of Mathematics} \textbf{7} (1936), 134–151. 
		
		\bibitem{kleidlieb90} Peter Kleidman and Martin Liebeck, \emph{The Subgroup Structure of the Finite Classical Groups}, Cambridge University Press, Cambridge, 1990. 
		
		\bibitem{lucchini10} Andrea Lucchini, On the subgroups with non-trivial M\"obius number, \emph{Journal of Group Theory}, \textbf{13} (2010), 589-600. 
		
		\bibitem{mann05} Avinoam Mann, A probabilistic zeta function for arithmetic groups, \emph{International Journal of Algebra and Computation}, \textbf{15} (2005), 1053-1059. 	
		
		\bibitem{schroder15} Anna Katharina Schr\"oder, \emph{The maximal subgroups of the classical groups in dimension 13, 14 and 15}, Ph.D. Thesis, University of St Andrews, 2015. 
		
		\bibitem{shareshian96} John Shareshian, \emph{Combinatorial properties of subgroup lattices of finite groups}, Ph.D. Thesis, Rutgers - The State University of New Jersey, 1996. 
		
		\bibitem{stanley86} Richard P. Stanley, \emph{Enumerative Combinatorics, Volume I}, second ed., Cambridge University Press, Cambridge, 2012 (first ed. published by Wadsworth \& Brooks/Cole, 1986).}

\end{thebibliography}
\end{document}